\documentclass[]{imsart}
\usepackage{amsmath}

\usepackage{graphicx}
\usepackage{color}
\usepackage{amsmath}
\usepackage{amssymb}
\usepackage{amscd}
\usepackage{bbm}

\usepackage{tikz}
\newlength\Colsep
\usetikzlibrary{patterns}

\usepackage{hyperref}
\hypersetup{
	bookmarks=true,         
	colorlinks=true,       
	linkcolor=red,          
	citecolor=green,        
	filecolor=magenta,      
	urlcolor=cyan           
}

\usepackage[utf8]{inputenc}
\usepackage{amsthm}
\usepackage{bbm} 
\usepackage[linesnumbered,lined,boxed,commentsnumbered]{algorithm2e}

\usepackage{marginnote}

\newtheorem{Theorem}{Theorem}

\newtheorem{Definition}{Definition}
\newtheorem{Lemma}{Lemma}
\newtheorem{Proposition}{Proposition}
\newtheorem{Corollary}{Corollary}

\newcommand{\inr}[1]{\bigl< #1 \bigr>}

\DeclareMathOperator*{\argmin}{argmin}

\def\ds1{\textrm{1\kern-0.25emI}} 

\newcommand \E{\mathbb{E}}
\newcommand \R{\mathbb{R}}

\newcommand \cC{{\cal C}}
\newcommand \cD{{\cal D}}

\newcommand \cL{{\cal L}}

\newcommand \cN{{\cal N}}

\newcommand \bE{{\mathbb E}}

\newcommand \bI{{\mathbb I}}

\newcommand \bP{{\mathbb P}}

\newcommand \bR{{\mathbb R}}

\newcommand \bX{{\mathbb X}}
\newcommand \bY{{\mathbb Y}}

\usepackage{bm}
\usepackage{comment}

\begin{document}
	
	\begin{frontmatter}
		
		\title{On the robustness of the minimim $\ell_2$ interpolator}
		
		\begin{aug}
			\author{\fnms{Geoffrey}  \snm{Chinot}\thanksref{a}\ead[label=e1]{geoffrey.chinot@stat.math.ethz.ch}}
			\and	 
			\author{\fnms{Matthieu}  \snm{Lerasle}\thanksref{b}\ead[label=e2]{matthieu.lerasle@ensae.fr}}

			
			\affiliation{ENSAE, CREST, Institut Polytechnique de Paris}
			
			\address[a]{ETHZ, Rämistrasse 101, 8092 Zürich.\\
			 \printead{e1}}
			 \address[b]{CNRS, ENSAE, CREST, 5 avenue Henri Chatelier 91120 Palaiseau, France.\\
			 \printead{e2}}			
		\end{aug}
		
		\begin{abstract}
		 We analyse the interpolator with minimal $\ell_2$-norm $\hat{\beta}$ in a general high dimensional linear regression framework where $\bY=\bX\beta^*+\xi$ where $\bX$ is a random $n\times p$ matrix with independent $\cN(0,\Sigma)$ rows and without assumption on the noise vector $\xi\in \R^n$.
		 We prove that, with high probability, the prediction loss of this estimator is bounded from above by $(\|\beta^*\|^2_2r_{cn}(\Sigma)\vee \|\xi\|^2)/n$, where $r_{k}(\Sigma)=\sum_{i\geq k}\lambda_i(\Sigma)$ are the rests of the sum of eigenvalues of $\Sigma$.
		 These bounds show a transition in the rates. For high signal to noise ratios, the rates $\|\beta^*\|^2_2r_{cn}(\Sigma)/n$ broadly improve the existing ones. For low signal to noise ratio, we also provide  lower bound holding with large probability. Under assumptions on the sprectrum of $\Sigma$, this lower bound is of order $\| \xi\|_2^2/n$, matching the upper bound. Consequently, in the large noise regime, we are able to precisely track the prediction error with large probability. This results give new insight when the interpolation can be harmless in high dimensions. 
		\end{abstract}
		
		
	\end{frontmatter}

\section{Introduction} 

In this paper, we consider the problem of estimating a vector $\beta^*\in \R^p$ from possibly noisy observations of random projections of it.
Let $\bX\in \R^{n\times p}$ denote a random matrix with rows $X_i^T$.
The observations can therefore equivalently be written 
\[
y_i=\inr{X_i,\beta^*}+\xi_i,\qquad i\in\{1,\ldots,n\}\enspace,
\]
or, in the matrix form 
\[
\bY=\bX\beta^*+\xi\enspace.
\]
The vector $\xi=(\xi_1,\ldots,\xi_n)^T$ is called the noise.
This problem is classical in signal processing and in statistics where it is known as the linear regression problem. 
In particular, the Gaussian linear regression problem is the problem of recovering $\beta^*$ when $\xi$ is independent from $\bX$ and Gaussian.
The arguably most famous estimator is least-squares estimator defined as 
\[
\hat{\beta}\in\argmin_{\beta\in \R^p}\|\bY-\bX\beta\|^2_2\enspace,
\]
where $\|\cdot\|_2$ denotes the usual Euclidean norm in $\R^d$, whatever $d\geqslant 2$.
The quality of $\hat{\beta}$ can be assessed through upper and lower bounds on the estimation error $\|\hat{\beta}-\beta^*\|_2$.
When the rows $X_i^T$ of $\bX$ are i.i.d. with second moment matrix $\Sigma=\E[X_1X_1^T]$, another popular quality measure is the prediction error 
\[
\|\Sigma^{1/2}(\hat{\beta}-\beta^*)\|_2=\sqrt{\E[\inr{X,\hat{\beta}-\beta^*}^2|\bX,\bY]}\enspace.
\]
In this formulation, $X$ denote an independent copy of $X_1$, independent from $\bX,\bY$. The prediction error therefore measures how far are the predictions at a typical point $X$ by $\hat{\beta}$: $\inr{X,\hat{\beta}}$ and by the actual signal $\beta^*$: $\inr{X,\beta^*}$.
Both risks are random variables and upper and lower bounds for these risks, in expectation and with high probability are now well known in the small dimensional Gaussian linear problem where $p<n$.

These bounds deteriorate as the dimension grows and the least-squares estimator behaves poorly when $p\asymp n$.
This can be understood as follows: In high dimension where $p\geqslant n$, the set of least-squares estimators is typically infinite. 
Actually, when the matrix $\bX$ has full rank $n$, its null space is non trivial and any solution in the set $\{\bX^g\bY\}$, where $\bX^g$ describes all pseudo-inverses of $\bX$ satisfy $\bX\bX^g\bY=\bY$.
In other words, in large dimension, least-squares estimators interpolate data.
This kind of behavior is typically undesirable in statistics, as the estimators clearly overfit the observed dataset, and have usually poor generalization abilities.
The least-squares estimators are not the only estimators suffering this kind of limitation, actually, this feature is shared with any estimator without further assumptions on $\beta^*$, a phenomenon known as the curse of dimensionality in statistics.

The classical trick in high dimensional statistics to bypass this issue is to assume structural assumptions on $\beta^*$, such as sparsity or regularity assumptions.
This has given rise to an impressive literature these last decades.
We cannot review here this massive literature. The interested reader can find comprehensive introductions to these topics in the textbooks \cite{buhlmann2011statistics,giraud2014introduction,vershynin2018high,wainwright2019high, fan2020statistical} and the references therein.
Let us just mention that this approach was proved efficient, for example in the high dimensional Gaussian regression problem. 
Among popular such algorithms, one can mention basis pursuit \cite{ChenDonohoSaunders98}, ridge regression~\cite{hoerl1970ridge,casella1980minimax}, the LASSO~\cite{tibshirani1996regression,van2008high,bickel2009simultaneous} and the elastic net~\cite{zou2005regularization,de2009elastic}.

We do not pursue this path in this paper, we want to tackle the problem in high dimension, that is when $p\geqslant n$, without any assumptions on $\beta^*$.
As we said, bounding the estimation error $\|\hat{\beta}-\beta^*\|_2$ in this context remains impossible in general~\cite{vershynin2015estimation}. However, and perhaps counter-intuitively, \cite{bartlett2019benign} discovered that, when the dimension $p$ is large in front of $n$, the prediction risk $\|\Sigma^{1/2}(\hat \beta-\beta^*)\|_2^2$ can be small for the least-squares estimator $\hat{\beta}=\bX^{\dagger}\bY$, where $\bX^{\dagger}$ is the Moore-Penrose pseudo-inverse of $\bX$.
They proved that this holds in the Gaussian regression problem, when the lines $X_i^T$ of $\bX$ are i.i.d. with Gaussian distribution $\cN(0,\Sigma)$, under some conditions on the spectrum of $\Sigma$.
An important take home message of \cite{bartlett2019benign} is therefore the following
\begin{center}
	\textbf{In high dimension, even without structural assumptions on $\beta^*$, it is still possible to predict well. }
\end{center}

This interesting phenomenon has given rise to a rapidly growing literature these last months, see~\cite{belkin2019reconciling,belkin2019two,belkin2018overfitting,belkin2018does, bunea2020interpolation, feldman2019does,liang2018just,mei2019generalization}.
This success is not surprising as many algorithms in machine learning require to fit a huge number of parameters with a smaller number of data.
The most famous examples are neural networks for which it has been repeatedly observed empirically that enlarging the network, hence, the number of parameters, may help to improve prediction performance~\cite{advani2017high,belkin2019reconciling,zhang2016understanding}. 
Of course, the linear prediction problem here is much simpler than understanding the predictions of neural networks, but it is interesting to understand when and how high dimension helps prediction. 
Moreover, several recent works have shown that the analysis of linear models can be relevant for over-parametrized neural networks, see for example \cite{chatterji2020does}. 
A reason is that, when neural networks are trained by gradient descent properly initialized, they are well approximated by a linear model in a Hilbert space. 
This method is known as~\emph{neural tangent kernel} approach~\cite{jacot2018neural, bietti2019inductive, arora2019exact, lee2019wide}. 
Understanding the generalization of over-parametrized linear models could therefore be seen as a first step in the direction of understanding deep learning.\\

In this paper, as in \cite{bartlett2019benign}, we analyse the least-squares estimator $\bX^{\dagger}\bY$, where $\bX^{\dagger}$ denotes the Penrose Moore inverse of $\bX$, which is the least-squares estimator with minimal Euclidean norm. 
We also assume that $\bX$ has i.i.d. Gaussian $\cN(0,\Sigma)$ lines.
We will assume all along the paper that we are in the high dimensional regime where $p \geq n$ and that $\Sigma$ has rank larger than $n$, which implies in particular that $\bX$ has a.s. full rank $n$. 
In this setting, the estimator can be equivalently defined as the minimum norm interpolator
\begin{equation} \label{def_interpolation}
\hat \beta = \argmin_{\beta \in \bR^p } \|\beta\|_2 \quad \textnormal{subject to}  \quad \bX \beta = \bY \enspace.
\end{equation}
The main difference with \cite{bartlett2019benign} is that we relax all assumptions on the noise $\xi$. 
This shows the robustness of $\hat{\beta}$ to various contaminations of the response data $\bY$, as the noise, for example, can be deterministic (and hence null), random with any distributions (hence allowing heavy tailed perturbations), or even adversarial. 
Similar assumptions have been considered in the compressed sensing community~\cite{Wojtaszczyk10}, in a different problem and with the Euclidean norm replaced by the $\ell_1$-norm.

Our main results give upper and lower bounds on the prediction risk of $\hat\beta$, $\|\Sigma^{1/2}(\hat \beta-\beta^*)\|_2^2$.
The bounds are typically dependent on the noise $\|\xi\|_2$ and are therefore random in general. 
To deduce deterministic bounds, an independent analysis of this term should be performed in each situation.
As in \cite{bartlett2019benign}, the bounds are interesting under assumptions on the spectrum of the covariance matrix $\Sigma$. 
These assumptions involve the rest of the series of singular values of the matrix $\Sigma$, $r_{k}(\Sigma)=\sum_{i=k}^{p}\lambda_i(\Sigma)$.
Our bounds exhibit a phase transition when the signal to noise ratio $\text{SNR}=\|\beta^*\|_2^2/ \| \xi \|_2^2$ becomes larger than a threshold $t=1/r_{cn}(\Sigma)$, where $c$ is an absolute constant.
\begin{itemize}
 \item For high signal to noise ratios: $\text{SNR}>t$, the prediction risk of the estimator satisfies, $\|\Sigma^{1/2}(\hat{\beta}-\beta^*)\|_2^2\lesssim \|\beta^*\|^2r_{cn}(\Sigma)/n$, with large probability. This result improves the one presented in~\cite{bartlett2019benign} in two ways: first they only reached in this regime the more pessimistic upper bound $\|\Sigma^{1/2}(\hat{\beta}-\beta^*)\|_2^2\lesssim \|\beta^*\|^2\sqrt{\text{Tr}(\Sigma)/n}$.
 Notice that, in a different framework the improvement on the rate from $\sqrt{1/n}$ to $1/n$ for interpolators already appeared in \cite{bunea2020interpolation}.
 Second, we prove that, for interpolators, these fast rates of convergence hold with probability $1-ce^{- n/c}$, for some absolute constant $c$, while in~\cite{bartlett2019benign} the results where established with probability $1-ce^{- c\textnormal{Tr}(\Sigma)}$.
 \item When the SNR is low, $\text{SNR}\leq t$, we show that with probability $1-ce^{- n/c}$,
$$
\| \Sigma^{1/2} (\hat \beta - \beta^*) \|_2^2 \leq c \frac{\| \xi \|_2^2}{n} \enspace.
$$
In Gaussian linear regression with $\xi\sim\cN(0,\sigma^2\bI)\in \R^n$, this rates becomes $\sigma^2$ with probability $1-ce^{-n/c}$, which matches the minimax rate of convergence for this confidence level \cite[Theorem A' ]{lecue2013learning}.
Besides, we prove that this bound cannot be improved in general as, when the noise in independent from $\bX$ (whatever its distribution), indeed, for a well chosen $\bar{k}\leqslant p$ (see Section~\ref{sec:MainResults} for a precise definition),
\[
\| \Sigma^{1/2} (\hat \beta - \beta^*) \|_2^2\geqslant c \frac{\| \xi \|_2^2}{n\wedge \bar{k} } \enspace.
\]
Therefore, in the particular case where $p = cn$ for example, this yields $\| \Sigma^{1/2} (\hat \beta - \beta^*) \|_2^2 \asymp \| \xi \|_2^2 / n$. 
More generally, for any spectrum of $\Sigma$, such that $\bar k \leq cn$ (we provide an example at the end of Section~\ref{sec:MainResults} where this holds while $n=o(p)$), upper and lower bounds match $\| \Sigma^{1/2} (\hat \beta - \beta^*) \|_2^2 \asymp \| \xi \|_2^2 / n$. 
Interestingly, the lower bounds are also obtained with probability $1-ce^{-n/c}$.
In this regime, the comparison with \cite{bartlett2019benign} is less clear.
On one hand, our result are more general as they allow any kind of noise and, when the noise is Gaussian, improve the bounds of \cite{bartlett2019benign} at confidence levels $1-ce^{-n/c}$. 
On the other hand, when the noise is Gaussian $\xi\sim \cN(0,\sigma^2\bI)$, our bounds do not shrink to $0$ as $n\to \infty$ while those in \cite{bartlett2019benign} might at smaller confidence levels. \\
Our lower bound depends on a new parameter $\bar k$ that was not present in the previous work of~\cite{bartlett2019benign}. It gives new insights when the overfitting can be harmless.  
\end{itemize}

Our extension to general noise is based on the following observation: As $\hat \beta$ interpolates, we have
	\begin{equation}\label{eq_deviations}
	\| \Sigma^{1/2} (\hat \beta - \beta^*) \|_2^2 +  \underbrace{\frac{1}{n} \sum_{i=1}^n \langle X_i, \hat \beta - \beta^* \rangle^2 - \| \Sigma^{1/2} (\hat \beta - \beta^*) \|_2^2}_{\textnormal{deviation}} = \frac{\| \xi \|_2^2}{n}  \enspace.
	\end{equation}
	The main technical contribution of the paper is then a control of the deviation term that holds independently to the form of the noise $\xi$.
%
	The control of the deviation term is possible using a preliminary results showing that dimension may help to localize this estimator with respect to the estimation norm $\|\hat \beta-\beta\|_2$. Heuristically, since the dimension of the set of interpolators increases with the dimension $p$, it is expected that $\| \hat \beta \|_2$ also decreases with the dimension. More precisely, we show that$\| \hat \beta - \beta^* \|_2 \leq \| \beta^* \|_2 + r(\Sigma)$, where  $r(\Sigma)$ is a remainder term controlling the improvement with the dimension $p$.  

The paper is divided in two parts, Section~\ref{sec:MainResults} presents the assumptions and main results, and discuss the general case in particular situations of interest. The main proofs are gathered in Section~\ref{sec:proofs}.
 
\paragraph{Notations}
 For any symmetric matrix $A \in \bR^{n \times n}$, we denote by $\lambda_1(A) \geq \cdots \geq \lambda_n(A)$ its eigenvalues in the non-increasing order and by $r_k(A)=\sum_{i=k}^n\lambda_i(A)$. 
More generally, for any matrix $B \in \bR^{n \times p}$, we denote by $\sigma_1(B) \geq \cdots \geq \sigma_{min}(B) > 0$, its positive singular values in the non-increasing order. 
The operator norm of $B$ is denoted by $\|B\| = \sigma_1(B)$. 
 For any symmetric positive semi-definite matrix $A$, let $\|\beta\|_A=\sqrt{\beta^TA\beta}$.
 Let $S(r)$ (resp. $S_A(r)$) denote the sphere in $\bR^p$ with radius $r$ with respect to the Euclidean norm $\|\cdot\|_2$ (resp. with respect to the semi-norm $\|\cdot\|_A$).
 Define similarly $B(r)$ and $B_A(r)$ to be the balls with radius $r$. 
 All along the paper, $c,c_1,c_2 \cdots$ denote absolute positive constants whose values may change from one instance to another.
 
\section{Main results}\label{sec:MainResults}
This section provides our main contributions. Before stating our main result, let us introduce quantities that will drive the prediction risk of $\hat \beta$.  
Define for an absolute constant $c_0$,
 \begin{equation}\label{eq_kstar}
\rho= \|\beta^* \|_2 + \frac{4 \| \xi \|_2 }{\sqrt{r_{k^*}(\Sigma)}},\qquad \text{where}\qquad k^* = \inf \bigg\{ k \in \{1,\cdots, p \}:  \frac{ r_k (\Sigma) }{ \lambda_k (\Sigma)}   \geq  c_0  n \bigg\} \enspace,
\end{equation} 
with the convention that $\inf\emptyset=+\infty$. Also, for constants $\eta,\gamma >0$, let us define the following two complexity parameters:
\begin{gather}
\label{def:r^*}   r^*(\eta) =   \inf \bigg\{ r>0 :  \sum_{i=1}^p \lambda_i(\Sigma) \wedge r^2  \leq  \eta n r^2 \bigg\} \\
\label{def:bar_r}   \bar r (\gamma) =   \sup \bigg\{ r>0 :  \sum_{i=1}^p \lambda_i(\Sigma)\rho^2 \wedge r^2  \leq  \gamma  \| \xi \|_2^2 \bigg\} \enspace.
\end{gather}
We are now in position to state our main theorem. 
\begin{Theorem} \label{theorem_loss_quadra} 
Assume $k^*\leq cn$, for $c >0$ an absolute constant. \\
For $\eta$ small enough, there exist absolute constants $c_1,c_2,c_3$  such that with probability larger than $1-c_1e^{-c_2n}$, the estimator $\hat \beta$ defined in Equation~\eqref{def_interpolation} satisfies
	\begin{gather*}
	\|\hat{\beta} - \beta^*\|_2 \leq \rho\,\quad \mbox{and} \quad 
	\|\Sigma^{1/2} (\hat{\beta} - \beta^*)\|_2 \leq  \rho r^*(\eta) \vee c_3 \frac{\| \xi \|_2}{\sqrt n}	 \enspace.
	\end{gather*}
	Moreover, let us assume that $X_i | \xi$ are i.i.d. $\cN(0,\Sigma)$ (recall that $X_i^T$ is the $i$-th row of $\bX$). 
	There exist absolute constants $\gamma>0$, $c_1,c_2,c_3$ such that with probability larger than $1-c_1e^{-c_2n}$, the estimator $\hat \beta$ defined in Equation~\eqref{def_interpolation} satisfies
	$$
	\|\Sigma^{1/2} (\hat{\beta} - \beta^*)\|_2\geq \bar r(\gamma) \wedge c_3\frac{\| \xi \|_2}{\sqrt n}\enspace.
	$$
\end{Theorem}
Theorem~\ref{theorem_loss_quadra} is proved in Section~\ref{sec:ProofThmPredQuad}.
The estimation bound $\rho$ does not converge to $0$. 
This is not surprising in high dimension without sparsity assumption. 
However, it is interesting to see that it may decrease, up to a certain threshold, with the dimension $p$. 
In particular, when the signal to noise ratio $\|\beta^*\|^2/\| \xi \|_2^2$ is larger that the threshold $1/r_{k^*}(\Sigma)$, $\|\hat \beta - \beta^*\|_2$ is at most of order $\|\beta^*\|_2$.

As stressed before, the upper bound in~Theorem~\ref{theorem_loss_quadra} holds without assumption on the $\xi_i$'s. They can be deterministic or even depend on $\bX$. 
This is a major difference with previous results in the literature such as~\cite{bartlett2019benign,bunea2020interpolation} where this noise was always Gaussian and independent from $\bX$. 
Here, the results can be applied to the following example, where $\bY$ is itself the output of a prediction algorithm, that is, when $\bY=f(\bX)=(f_1(\bX_1),\ldots,f_n(\bX_n))^T$. 
In this case, the upper bounds becomes
 $$
 \frac{\| \xi \|^2_2}{n}=\frac{\|f(\bX)-\bX\beta^*\|_2^2}n = \frac{1}{n} \sum_{i=1}^n \big( \langle \bX_i, \beta^* \rangle  - f_i (\bX_i) \big)^2 \enspace.
 $$
This error measures how far the initial prediction is from the linear model.\\

To discuss the upper prediction bounds, it is useful to give the following corollary. It shows a phase transition in the prediction rates when the signal to noise ratio $\text{SNR}=\|\beta^*\|^2/\| \xi \|^2$ becomes larger than the threshold $t=1/r_{cn}(\Sigma)$.

\begin{Corollary}\label{cor:CompBartlett}
Grant the assumptions and notations of Theorem~\ref{theorem_loss_quadra}. The estimator $\hat \beta$ defined in Equation~\eqref{def_interpolation} satisfies, with probability larger than $1-c_1e^{-c_2n}$,
		\begin{gather*}
		\|\Sigma^{1/2} (\hat{\beta} - \beta^*)\|_2^2  \lesssim \|\beta^* \|_2^2 \frac{r_{cn}(\Sigma)}{n} \vee \frac{\| \xi \|_2^2}{n} \enspace.
		\end{gather*}
\end{Corollary}

Corollary~\ref{cor:CompBartlett} is proved in Section~\ref{sec:ProofCorollaries}. 
It can be used to compare our results with \cite{bartlett2019benign}.
\begin{enumerate}
 \item  If the signal to noise ratio is large enough, $\|\beta^* \|_2^2/\| \xi \|_2^2\geq 1/r_{cn}(\Sigma)$ the bounds in \cite{bartlett2019benign} are always larger than ours. Indeed, for large SNR, our bounds are of order $\|\beta^* \|_2^2 r_{cn}(\Sigma)/n$ while theirs are of order $\|\beta^* \|_2^2 \sqrt{\textnormal{Tr}(\Sigma)/n} $. 
The improvement can even be exponential improvements as shown in the example below. 
  \item For small signal to noise ratios, $\text{SNR}<t$, our prediction rates are of order $\frac{\| \xi \|_2}{\sqrt n}$. This bound holds without any assumption on the noise $\xi$. 
  Contrary to~\cite{bartlett2019benign}, this rate does not converge to $0$ as $n \rightarrow \infty$ when $\xi \sim \cN(0,\sigma^2\bI)$ is independent from $\bX$. 
  However, in this case, $\| \xi \|_2^2 / n \asymp \sigma^2$ matches the optimal rate holding with probability larger than $1-c_1\exp(-c_2n)$ (see \cite[Theorem A' ]{lecue2013learning}). 
\end{enumerate}
To illustrate their upper bounds, \cite{bartlett2019benign} provide several examples of ``benign matrices" where the different quantities of interest in Theorem~\ref{theorem_loss_quadra} can easily be computed.
We compute the quantities appearing in one of these examples now.\\
Assume that there exist $\epsilon>0$ (small), $c>0$ and $\tau$ such that $p=cn$, $\tau\log(1/\epsilon) \leq n$ and
\[
\forall k\geq 1,\qquad \lambda_k(\Sigma)=e^{-k/\tau}+\epsilon \enspace.
\]
In this case, for any $k$,
\begin{gather*}
\frac{r_k}{\lambda_k} \asymp \frac{(p-k)\epsilon+ \tau \exp(-k/ \tau)}{e^{-k/\tau}+\epsilon}\enspace,\\
\end{gather*}
Therefore, for $k=\tau\log(1/\epsilon)$,
\begin{gather*}
 \frac{r_k}{\lambda_k}\gtrsim\frac{p\epsilon+\tau \epsilon}{\epsilon} \gtrsim  n,\qquad \text{so}\qquad k^* \leq \tau\log(1/\epsilon)\leq n\enspace.
\end{gather*}
Moreover, $r_{cn}(\Sigma) \lesssim p\epsilon + \tau \exp(-cn /\tau)$  and Corollary~\ref{cor:CompBartlett} shows in this example that with probability larger than $1-c_1e^{-c_2n}$, 
\[
\|\Sigma^{1/2} (\hat{\beta} - \beta^*)\|_2^2  \lesssim \|\beta^* \|_2^2 \frac{p \epsilon+\tau \exp(-n / \tau)}{n}  \vee \frac{\| \xi \|_2^2}{n}\enspace.
\]
Our rates of convergence in this example can therefore be as fast as $\epsilon\vee \tau \exp(-n/\tau )/n$, while \cite[Theorem 6]{bartlett2019benign} gives in this setting a rate $\sqrt{ \epsilon \vee \tau \exp(-1/\tau)/  n}$ leading to a potential exponential improvement (a multiplication by $e^{-c n}$ of the rates) for small $\epsilon$. \\

Let us turn to the study of lower bound in the large noise regime. 
\begin{Corollary}\label{cor:lower_bound}
	Grant the assumptions and notations of Theorem~\ref{theorem_loss_quadra}, with $\gamma$ chosen such that the conclusion of the Theorem holds.
	Assume that $X_i | \xi \sim \cN(0,\Sigma)$ and are independent conditionally on $\xi$. 
	If the signal to noise ratio is small, $\|\beta^* \|_2^2/\| \xi \|_2^2\leq 1/r_{k^*}(\Sigma)$, then, the estimator $\hat \beta$ defined in Equation~\eqref{def_interpolation} satisfies, with probability larger than $1-c_1e^{-c_2n}$,
		$$
		 \|\Sigma^{1/2} (\hat{\beta} - \beta^*)\|_2^2\geq c_3\frac{\| \xi \|_2^2}{ n \wedge \bar k } \enspace,
		$$
		where
	 \begin{equation}\label{eq_kbar}
	\bar k = \inf \bigg\{ k \geq k^* :  \sum_{i=k}^p \lambda_i (\Sigma)   \leq  \frac{\gamma}{2}   \sum_{i=k^*}^p \lambda_i (\Sigma)\bigg\} \enspace,
	\end{equation} 
	with the convention that $\bar{k}=p+1$ if the set is empty.
\end{Corollary}
Corollary~\ref{cor:lower_bound} is proved in Section~\ref{sec:ProofCorollaries}.
Note that when $p = cn$, we have $\bar k \leq cn$. In this case, in the large noise regime,
$$
\|\Sigma^{1/2} (\hat{\beta} - \beta^*)\|_2^2 \gtrsim  \frac{\| \xi \|_2^2}{n}	 \enspace,
$$
with probability larger than $1-c_1\exp(-c_2n)$, which matches the upper bound. \\

Let us give an example of spectrum where $p$ might be much larger than $n$ while $\bar k$ remains smaller than $n$. Let $k_1$, $c$, $k_2=k_1+cn+1$, $\varepsilon_1,\varepsilon_2$ real numbers and assume that
$$
\lambda_i = \left\{
\begin{array}{ll}
1 & \mbox{if } i \leqslant k_1-1\enspace,\\
\varepsilon_1 & \mbox{if } i \in \{k_1 \cdots, k_1-1 \} \enspace,\\
\varepsilon_2 & \mbox{if } i \geqslant k_2\enspace.
\end{array}
\right.
$$
In this case,
\begin{gather*}
 \frac{r_{k_1}}{\lambda_{k_1}} \geqslant cn + (p-k_1-cn) \frac{\varepsilon_2}{\varepsilon_1} \geq cn\enspace,\\
\end{gather*}
It follows that $k^*\leqslant k_1$. Finally 
$$
\frac{\sum_{i= k_2}^p  \lambda_i}{\sum_{i= k_1}^p  \lambda_i} \leqslant \frac{(p-k_2 +1) \varepsilon_2}{cn\varepsilon_1 + (p-k_2 +1)\varepsilon_2} \enspace.
$$ 
This ratio is smaller than $\gamma/2$ if 
$$
\varepsilon_2 \leq  \frac{\gamma}{2-\gamma}   \frac{cn}{p- k_2 +1} \varepsilon_1 \enspace.
$$
This proves that $\bar{k}\leqslant k_1+cn$. 
If $k_1\leq n$, we have therefore, in this example, if  $\|\beta^* \|^2_2 / \| \xi \|_2^2 \leq r_{cn}(\Sigma)$
$$
\|\Sigma^{1/2} (\hat{\beta} - \beta^*)\|_2^2 \gtrsim \frac{\| \xi \|_2^2}{n}\enspace.
$$

\section{Proofs of the main results}\label{sec:proofs}
The remaining of the paper is devoted to the proofs of the main results.
Section~\ref{sec:ProofEstBounds} (resp.~\ref{sec:ProofThmPredQuad}) shows the estimation bound (resp. the prediction bounds) in Theorem~\ref{theorem_loss_quadra}. 

\subsection{Proof of the estimation bound of Theorem~\ref{theorem_loss_quadra}}\label{sec:ProofEstBounds}
The following theorem establishes the bound on the estimation error in Theorem~\ref{theorem_loss_quadra}.
In the following section, this preliminary estimate will be used to ``localize" the analysis of the prediction risk of $\hat \beta$.
This approach is now classical in statistical learning, it has been applied successfully, for example, in~\cite{MR3431642, mendelson2014learning, mendelson2016upper, mendelson2017multiplier}.

\begin{Theorem} \label{theorem_control_reg_norm}
There exists $c>0$ such that, if $c_0\geq c$ in the definition of $k^*$, the estimator $\hat \beta$ defined in Equation~\eqref{def_interpolation} satisfies
\begin{equation}\label{eq_proba_reg}
\bP \left(  \|\hat \beta - \beta^* \|_2 \leq \| \beta^* \|_2 +\frac{4\| \xi \|_2 }{ \sqrt {r_{k*}(\Sigma)}}  \right) \geq 1 - 2\exp\big(-n\big) \enspace.
\end{equation}
\end{Theorem}

\begin{proof}[Proof of Theorem~\ref{theorem_control_reg_norm}]
The proof starts with the following lemma.
\begin{Lemma} \label{lemma_start}
	The estimator $\hat \beta$ verifies 
	\begin{equation} \label{eq:lemma1}
	\| \hat \beta - \beta^* \|_2 \leq  \|\beta^* \|_2 +  \frac{\|  \xi \|_2}{\sigma_n( \bX  )} \enspace.
	\end{equation}
\end{Lemma}

\begin{proof}[Proof of Lemma~\ref{lemma_start}]
	Recall that 
	\[
	\hat \beta =\bX^{\dagger}\bY= \bX^{\dagger} \bX\beta^* + \bX^{\dagger}  \xi \enspace,
	\]
	where $\bX^{\dagger}$ denotes the Moore-Penrose pseudo-inverse of $\bX$. 
Therefore,
	\begin{equation} \label{eq_start}
	\| \hat \beta - \beta^* \|_2 = \| (\bX^{\dagger} \bX - I_p) \beta^* - \bX^{\dagger}  \xi \|_2 \leq \|\beta^*\|_2 + \|\bX^{\dagger} \xi \|_2\enspace,
	\end{equation} 
	where the last inequality follows from the triangular inequality and the fact that $\bX^{\dagger} \bX -I_p$ is the projection matrix onto the null-space of $\bX$. 
	Since $\| \bX^{\dagger}  \xi \|_2 \leq \| \bX^{\dagger}\|  \| \xi\|_2 $ it follows that
	$$
	\| \hat \beta - \beta^* \|_2  \leq \|\beta^*\|_2 +  \|  \xi \|_2 \|\bX^{\dagger} \| = \|\beta^*\|_2 + \frac{\|  \xi \|_2}{\sigma_n( \bX  )} \enspace,
	$$
	where the last identity holds because rank$(\bX) = n$ a.s. and thus $\| \bX^{\dagger} \| = \sigma_n^{-1}(\bX)$. 
\end{proof}

Lemma~\ref{lemma_start} provides a random bound on the estimation error of $\hat \beta$. 
To prove Theorem~\ref{theorem_control_reg_norm}, it remains to bound from below, with high probability, the $n$-th singular value $\sigma_n( \bX  )$ of $\bX$.
This control is obtained in the following lemma.
\begin{Lemma} \label{lemma_operator_norm}
	With probability larger than $
	1 -   2\exp (-n)$,
	if $c_0$ in the definition \eqref{eq_kstar} of $k^*$ is large enough, we have 
	\begin{equation*}
	\sigma_n(\bX) \geq  \frac{ \sqrt{r_{k^*}(\Sigma)}}{4}  \enspace.
	\end{equation*}
\end{Lemma}

\begin{proof}
The matrix $\bX^T$ is distributed as $\Sigma^{1/2} G$, where $G \in \bR^{p \times n}$ is a random matrix with i.i.d standard Gaussian variables, hence $\sigma_n( \bX ) = \sigma_n( \bX^T ) $ is distributed as $\sigma_n(\Sigma^{1/2} G ) $. 
Let $S^{n-1}$ denote the unit sphere in $\R^n$. 
From the Courant-Fischer-Weyl min-max principle, we have
	\begin{equation*}
	\sigma_n(\Sigma^{1/2} G  ) = \min_{x \in S^{n-1}} \|\Sigma^{1/2} Gx \|_2\enspace.
	\end{equation*}
	Let $\Lambda = \textnormal{diag}(\lambda_1(\Sigma),\cdots,\lambda_p(\Sigma))$.
	By the spectral theorem, there exists an orthogonal matrix $P$ such that $\Sigma^{1/2}=P\Lambda^{1/2}P^T$, so, for any $x\in S^{n-1}$, $\|\Sigma^{1/2}Gx\|_2^2=\| P \Lambda^{1/2}P^T Gx \|_2^2=\|\Lambda^{1/2}P^T Gx \|_2^2$.
Hence, by rotation invariance of Gaussian random vectors, $\|\Sigma^{1/2}Gx\|_2^2=\sum_{i=1}^p \lambda_i(\Sigma) g_i^2$, where $g_1, \cdots, g_p$ are i.i.d standard Gaussian random variables.
It follows that 
\begin{equation} \label{eq lower bound singular value}
\|  \Sigma^{1/2}G x \|_2^2 = \sum_{i=1}^p \lambda_i(\Sigma) g_i^2 \geq  \sum_{i=k^*}^p \lambda_i(\Sigma) g_i^2 = \| \Lambda_{k^*}^{1/2}  G x \|_2^2\enspace,
\end{equation}
where $k^*$ is defined in Theorem~\ref{theorem_loss_quadra} and $\Lambda_{k^*} = \textnormal{diag}(0,\cdots,0, \lambda_{k^*}(\Sigma), \cdots, \lambda_{^p}(\Sigma))$.

Let $\varepsilon \in (0,1)$ and $\mathcal N_{\varepsilon}$ be an $\varepsilon$-net of $S^{n-1}$.
A classical volume argument (see for example \cite[Corollary 4.2.13]{vershynin2018high}) shows that we can choose $\mathcal N_{\varepsilon}$ with $|\mathcal N_{\varepsilon}|\leq (3/\varepsilon)^n$. 
For any $x\in S^{n-1}$, there exists $y\in N_{\varepsilon}$ such that $\|x-y\|_2\leq \varepsilon$, so
\[
\|\Sigma^{1/2} Gx \|_2\geq \| \Lambda_{k^*}^{1/2}  G x \|_2\geq \|\Lambda_{k^*}^{1/2}  Gy \|_2-\varepsilon\|\Lambda_{k^*}^{1/2}  G \|\enspace.
\]
Hence,
\[
\sigma_n(\Sigma^{1/2} G  )\geq \sigma_n(\Lambda_{k^*}^{1/2}  G)\geqslant \min_{y\in N_{\varepsilon}} \|\Lambda_{k^*}^{1/2}  Gy \|_2-\varepsilon\|\Lambda_{k^*}^{1/2}  G \|\enspace.
\]	
Besides (see for example \cite[Lemma 4.4.1]{vershynin2018high})
\[
\|\Lambda_{k^*}^{1/2}  G \|\leq \frac1{1-\varepsilon}\max_{y\in N_{\varepsilon}} \|\Lambda_{k^*}^{1/2}  Gy \|_2\enspace,
\]
so 
\begin{equation}\label{eq spectral norm}
 \sigma_n(\Sigma^{1/2} G  )\geq\min_{y\in N_{\varepsilon}} \|\Lambda_{k^*}^{1/2}  Gy \|_2-\frac{\varepsilon}{1-\varepsilon}\max_{y\in N_{\varepsilon}} \|\Lambda_{k^*}^{1/2}  Gy \|_2\enspace.
\end{equation}

%

	Elementary computations show that, for any $i$, $\lambda_i(\Sigma) g_i^2$ is sub-exponential (see Definition~\ref{def_subexp} ) with parameters $(2\lambda_i(\Sigma),4\lambda_i(\Sigma))$.
	As these variables are independent, by Proposition~\ref{proposition_sum_subexp}, $\sum_{i=k^*}^p \lambda_i(\Sigma) g_i^2$ is sub-exponential with parameters $\big( 2 \big( \sum_{i=k^*}^p \lambda_i^2(\Sigma) \big)^{1/2},4 \lambda_{k^*}(\Sigma) \big)$. 
Therefore, by Proposition~\ref{proposition_tail_subexp}, with probability $1 - 2\exp (- t)$,
\begin{align}
\nonumber \big| \sum_{i=k^*}^p \lambda_i(\Sigma) g_i^2 - r_{k^*}(\Sigma) \big| &  \leq \max \bigg(\sqrt{8t  \sum_{i=k^*}^p \lambda_i^2(\Sigma)} , 8t \lambda_{k^*}(\Sigma) \bigg) \\
&\nonumber \leq \max \bigg(\sqrt{8\lambda_{k^*}(\Sigma) \  \sum_{i=k^*}^p \lambda_i(\Sigma)} , 8t \lambda_{k^*}(\Sigma) \bigg) \\
& \leq  \frac{r_k^*(\Sigma)}{2} + 12 t\lambda_{k^*}(\Sigma) \label{control subexp}      \enspace,
\end{align}
where we used the inequality $\sqrt{ab} \leq a/2 + b/2$, for any $a,b >0$ to get the last expression. 
A union bound shows therefore that, with probability $1-2\exp(-t+n\log(3/\epsilon))$,
	\[
	\frac{r_k^*(\Sigma)}{2} - 12 t\lambda_{k^*}(\Sigma)\leq\min_{y\in N_{\varepsilon}} \|\Lambda_{k^*}^{1/2}  Gy \|_2\leq \max_{y\in N_{\varepsilon}} \|\Lambda_{k^*}^{1/2}  Gy \|_2\leq \frac{3r_k^*(\Sigma)}{2} - 12 t\lambda_{k^*}(\Sigma)\enspace.
	\]
	Plugging this bound into \eqref{eq spectral norm} yields
	$$
		\sigma_n(\Sigma^{1/2} G x ) \geq  \bigg(1-\frac{\sqrt{3}\varepsilon}{1-\varepsilon}\bigg)\sqrt{\frac{r_{k^*}(\Sigma)}{2}} -  2\bigg(1+\frac{\sqrt{3}\varepsilon}{1-\varepsilon}\bigg)\sqrt{3 t\lambda_{k^*}(\Sigma)} \enspace.
	$$
	For $\varepsilon=1/4$, $t=n(1+\log(3/\varepsilon))$, this yields the result if $c_0$ in the definition \eqref{eq_kstar} of $k^*$ is large enough.
\end{proof}	
Theorem~\ref{theorem_control_reg_norm} then follows directly from Lemmas~\ref{lemma_start} and~\ref{lemma_operator_norm}.

\end{proof}

\subsection{Proof of the prediction bound in Theorem~\ref{theorem_loss_quadra} }\label{sec:ProofThmPredQuad}

We start this section with two Lemmas. Lemma~\ref{generic_deviation} enables to control the deviation of a quadratic process uniformly over a subset of $B_{\Sigma}(r)$, $r>0$. The main quantity driving this deviation is the Gaussian mean width that we introduce now. For any set $\mathcal B \subset \bR^p$ we define the Gaussian mean width of $\mathcal B$ as
\begin{equation} \label{def_gaussian_mean}
w(\mathcal B) = \bE   \sup_{\beta \in \mathcal B }  \langle \beta , g \rangle, \quad g \sim \cN(0,I_p) \enspace.
\end{equation}
The Gaussian mean width serves as a measure of effective dimension of the set $\mathcal B$ (see~\cite{amelunxen2014living} for equivalent formulations). \\
\begin{Lemma} \label{generic_deviation}
Let $r,\rho >0$ and $\delta \geq \exp(-n)$. Define $H_{r , \rho} = B(r) \cap B_{\Sigma^{-1/2}}(\rho)$,  where we recall that $B_{\Sigma^{-1/2}}(\rho)  = \{ \beta \in \bR^p : \sqrt{\beta^T \Sigma^{-1} \beta} \leq \rho  \}$ and $B(r) = \{  \beta \in \bR^p : \sqrt{\beta^T  \beta} \leq r \}$. There exists an absolute constant $c >0$ such that with probability larger than $1-\delta$
\begin{align} \label{empr_process}
 \sup_{\substack{ \| \beta - \beta^*\|_2 \leq \rho \\  \| \Sigma^{1/2} (\beta-\beta^*)\| _2 \leq r }}  \bigg| \frac{1}{n} \sum_{i=1}^n \langle X_i, \beta - \beta^* \rangle^2  & -  \bE \langle X_i, \beta- \beta^* \rangle^2 \bigg|\\
 & \leq c \left(  \frac{w^2 \big(  H_{r , \rho} \big)}{n} + r \sqrt \frac{w^2 \big(H_{r , \rho}  \big)}{n} + r^2 \sqrt{ \frac{\log(1/\delta)}{n}}   \right) \enspace,
\end{align}
\end{Lemma}

\begin{proof}
	Let $\delta \geq \exp(-n)$, $\mathcal B \subset \bR^p$ and $r>0$. Let $(g_i)_{i=1}^n$ be $n$ i.i.d centered standard Gaussian vectors in $\bR^p$. From~\cite[Theorem 5.5]{dirksen2015tail}, there exists an absolute constant $c>0$ such that with probability larger than $1-\delta$
	\begin{multline} \label{gene_chaining_Dirk}
	\sup_{\beta  \in B(r) \cap \mathcal B   } \bigg| \frac{1}{n} \sum_{i=1}^n \langle g_i, \beta \rangle^2  - \bE \langle g_i, \beta \rangle^2 \bigg| \\
	\leq c \left(  \frac{w^2( B(r) \cap \mathcal B)}{n} + r \sqrt \frac{w^2(B(r) \cap \mathcal B)}{n} + r^2 \sqrt{ \frac{\log(1/\delta)}{n}}   \right)\enspace.
	\end{multline}
	The rest of the proof simply consists in rewritting the empirical process~\eqref{empr_process}. Since $X_i $ is distributed as $\Sigma^{1/2}g_i$, where $g_i \sim \cN(0,I_p)$ it follows that
	\begin{multline*}  
	\sup_{\substack{ \| \beta - \beta^*\|_2 \leq \rho \\  \| \Sigma^{1/2} (\beta-\beta^*)\| _2 =r }}  \bigg| \frac{1}{n} \sum_{i=1}^n \langle X_i, \beta - \beta^* \rangle^2  -  \bE \langle X_i, \beta- \beta^* \rangle^2 \bigg| \\
	= \sup_{\substack{ \| \Sigma^{-1/2}\beta \|_2 \leq \rho \\  \| \beta \| _2 =r }}  \bigg| \frac{1}{n} \sum_{i=1}^n \langle g_i, \beta  \rangle^2  - \bE \langle g_i, \beta \rangle^2 \bigg| \enspace.
	\end{multline*}
	Applying~\eqref{gene_chaining_Dirk} to the right-hand term concludes the proof.
\end{proof}

\begin{Lemma} \label{gaussian_mean_width}
	Let $r,\rho >0$. The following holds
	\begin{equation} \label{gmw_ellip}
	w \big(   B(r) \cap B_{\Sigma^{-1/2}}(\rho)  \big) \leq \sqrt{2}  \big(  \sum_{i=1}^p  \lambda_i(\Sigma) \rho^2 \wedge  r^2  \big)^{1/2}  \enspace.
	\end{equation}
\end{Lemma}

\begin{proof}
From Equation~\eqref{def_gaussian_mean}, we have, for $g \sim \cN(0,I_p)$
\begin{align*}
w \big(  B(r) \cap B_{\Sigma^{-1/2}}(\rho)  \big)=   \bE \sup_{\beta  \in B(r) \cap B_{\Sigma^{-1/2}}(\rho)} \inr{g,t}  \enspace. 
\end{align*}
Moreover 
\begin{align*}
B(r) \cap B_{\Sigma^{-1/2}}(\rho)  &= \{  \beta \in \bR^p : \|\beta\|_2 \leq r,\|\Sigma^{-1/2} \beta \|_2 \leq  \rho \} \\
& = \bigg\{ \beta \in \R^p:  \sum_{i=1}^p\frac{\beta _i^2}{\lambda_i(\Sigma)\rho^2}\leq 1,\ \sum_{i=1}^p\frac{\beta_i^2}{r^2} \leq 1 \bigg\}\\
&\subset\bigg\{ \beta \in \R^p: \sum_{i=1}^p\frac{\beta_i^2}{\lambda_i(\Sigma)\rho^2\wedge r^2}\leq 2 \bigg\}\enspace.
\end{align*}
The Gaussian mean-width of an ellipsoid is given by~\cite[Proposition 2.5.1]{talagrand2014upper} and it follows that
\begin{equation} \label{gmw_ellip}
w \big(  B(r) \cap B_{\Sigma^{-1/2}}(\rho)  \big) \leq \sqrt{2}  \big(  \sum_{i=1}^p  \lambda_i(\Sigma) \rho^2 \wedge  r^2  \big)^{1/2}  \enspace.
\end{equation}
\end{proof}

\begin{Theorem} \label{lower_bound}
	For $\gamma >0 $, let us define  
	\begin{align*}
	 \bar r(\gamma) = \sup \bigg \{r >0 :  \sum_{i=1}^p  \lambda_i(\Sigma) \rho^2 \wedge  r^2 \leq \gamma \| \xi \|_2^2 \bigg\}  \enspace. 
	\end{align*}
	Assume that $X_i | \xi \sim \cN(0,\Sigma)$.   
	There exist $c_1,c_2,c_3 >0$ such that if $\gamma$ is small enough, then with probability larger than $1-c_1\exp(-cn_2) $
	$$
\bar r^2(\gamma) \wedge  c_3 \frac{\| \xi \|_2^2}{n}  \leq 	\| \Sigma^{1/2} (\hat \beta - \beta^* ) \|_2^2  \enspace.
	$$
\end{Theorem}

\begin{proof}
	Let $r = \bar r^2(\gamma) \wedge  c_3 \frac{\| \xi \|_2^2}{n} $ and define
	\begin{align} \label{def_omega_1}
	& \Omega_1 = \{   \| \hat \beta - \beta^* \|_2 \leq \rho  \}, \quad \rho  = \| \beta^* \|_2 + \frac{4\| \xi \|_2}{\sqrt{r_{k^*}(\Sigma)}}\enspace.
	\end{align}
	From Theorem~\ref{theorem_control_reg_norm}, the event $\Omega_1$ holds with probability larger than $1- 2\exp(-n)$.  Until the end of the proof, we place ourselves on the event $\Omega_1$. Since $\hat \beta$ is an interpolator, we have $\bX ( \hat \beta - \beta^*) = \xi$ and it follows that
	\begin{equation} \label{start_point_theorem}
	\frac{1}{n} \sum_{i=1}^n \langle X_i, \hat \beta - \beta^* \rangle^2 =  \frac{\| \xi \|_2^2}{n} \enspace.
	\end{equation}
	Assume that $\| \Sigma^{1/2}  (\hat \beta  - \beta^* ) \|_2 \leq r$. On $\Omega_1$, conditionally on $\xi$, we have
	\begin{align*}
	\frac{\| \xi \|_2^2}{n}  & \leq \sup_{\substack{ \| \beta  - \beta^* \|_2 \leq \rho \\ \| \Sigma^{1/2}(\beta - \beta^*) \|_2  \leq r}} \frac{1}{n} \sum_{i=1}^n \langle X_i,  \beta - \beta^* \rangle^2 \\
	& \leq r^2 + \sup_{\substack{ \| \beta  - \beta^* \|_2 \leq \rho \\ \| \Sigma^{1/2}(\beta - \beta^*) \|_2  \leq r   }} \bigg | \frac{1}{n} \sum_{i=1}^n \langle X_i,  \beta - \beta^* \rangle^2 - \bE \langle X_i,  \beta - \beta^* \rangle^2 \bigg| \\
	& \leq  r^2 +  c \left(  \frac{w^2 \big(  H_{r , \rho} \big)}{n} + r \sqrt \frac{w^2 \big(H_{r , \rho}  \big)}{n} + r^2 \sqrt{ \frac{\log(1/\delta)}{n}}   \right) \enspace,
	\end{align*}
	where the last inequality holds with probabiliy larger than $1-\delta$, for $\delta \geq \exp(-n)$, according to Lemma~\ref{generic_deviation}. 
	Now, from Lemma~\ref{gaussian_mean_width} and using the inequality $\sqrt{ab}\leq a/2+b/2$ for $a,b>0$ we obtain 
	 \begin{align*}
	 \frac{\| \xi \|_2^2}{n}  &   \leq   c \bigg[ r^2 \left( 1+ \frac{\log(1/\delta)}{n}   \right)  \bigg) +    \frac{w^2 \big(  H_{r , \rho} \big) }{n} \bigg] \\
 	& \leq c \bigg[ r^2 \left( 1+ \frac{\log(1/\delta)}{n}   \right) +  2\frac{ \sum_{i=1}^p  \lambda_i(\Sigma) \rho^2 \wedge  r^2 }{n} \bigg]\\
 	& \leq  cc_3 \frac{\| \xi \|_2^2}{n} \left( 1+ \frac{\log(1/\delta)}{n}   \right) + 2c\gamma   \frac{\| \xi \|_2^2}{n} \enspace,
	 \end{align*}
	where the last inequality holds because of the definition of $\bar r(\gamma)$. Taking $\delta = \exp(-n)$  and $\gamma=1/(4c)$ leads to a contradiction for $c_3$ small enough. 
\end{proof}

\begin{Theorem} \label{upper_bound}
	For $\eta >0 $, let us define  
	\begin{align*}
	r^*(\eta) = \inf \bigg \{r >0 :  \sum_{i=1}^p  \lambda_i(\Sigma)  \wedge  r^2 \leq \eta n r^2  \bigg\}  \enspace.
	\end{align*}
	There exists $c_1,c_2,c_3 >0$ such that if $\eta$ is small enough, then with probability larger than $1-c_1\exp(-cn_2) $,
	$$
		\| \Sigma^{1/2} (\hat \beta - \beta^* ) \|_2^2  \leq   \big(\rho r^*(\eta)\big)^2  \vee c_3 \frac{\| \xi \|_2^2}{n} \enspace. 
	$$
\end{Theorem}

\begin{proof}
	Until the end of the proof, we place ourselves on the event $\Omega_1$, defined in~\eqref{def_omega_1}. The proof is splitted in two parts.  \\
	
	1) Consider first the case where $\| \Sigma^{1/2}(\hat \beta - \beta^* ) \|_2 \leq r^*(\eta) \| \hat \beta - \beta^* \|_2$.\\
	Then, on $\Omega_1$, it follows that $\| \Sigma^{1/2}(\hat \beta - \beta^* ) \|_2 \leq  r^*(\eta) \rho$ so the conclusion of Theorem~\ref{upper_bound} holds.\\
	
	2) Now, consider the case where $\| \Sigma^{1/2}(\hat \beta - \beta^* ) \|_2 \geq r^*(\eta) \| \hat \beta - \beta^* \|_2$.\\
	 From~\eqref{start_point_theorem}, we have
	$$
	\frac{\| \xi \|_2^2}{n} =  \frac{\| \Sigma^{1/2}(\hat \beta - \beta^* ) \|_2^2}{n r^*(\eta)^2} \sum_{i=1}^n \bigg< X_i, r^*(\eta)\frac{\hat \beta - \beta^*}{\| \Sigma^{1/2}(\hat \beta - \beta^* ) \|_2} \bigg>^2 \enspace.
	$$
	Let us define $\tilde \beta -\beta^*= r^*(\eta) (\hat \beta - \beta^* )/ \| \Sigma^{1/2}(\hat \beta - \beta^* ) \|_2$. Since,  $\| \Sigma^{1/2}(\hat \beta - \beta^* ) \|_2 \geq r^*(\eta) \| \hat \beta - \beta^* \|_2$, we have $\| \tilde \beta - \beta^* \|_2 \leq 1$ and $\| \Sigma^{1/2}(\tilde \beta - \beta^*) \|_2 = r^*(\eta) $  and it follows that 
	\begin{align}
\notag	\frac{\| \xi \|_2^2}{n}&=\frac{\| \Sigma^{1/2}(\hat \beta - \beta^* ) \|_2^2}{nr^*(\eta)^2}     \sum_{i=1}^n \langle X_i,  \tilde\beta - \beta^* \rangle^2\\
\label{eq:Int1}	&\geq \frac{\| \Sigma^{1/2}(\hat \beta - \beta^* ) \|_2^2}{r^*(\eta)^2} \inf_{\substack{ \| \beta  - \beta^* \|_2 \leq 1 \\ \| \Sigma^{1/2}(\beta - \beta^*) \|_2  = r^*(\eta)  }} \frac{1}{n}   \sum_{i=1}^n \langle X_i,  \beta - \beta^* \rangle^2  \enspace.
	\end{align}
	Morover, we have 
	\begin{align}
\notag	 \inf_{\substack{ \| \beta  - \beta^* \|_2 \leq 1 \\ \| \Sigma^{1/2}(\beta - \beta^*) \|_2  = r^*(\eta)  }} & \frac{1}{n}   \sum_{i=1}^n \langle X_i,  \beta - \beta^* \rangle^2 \\ 
\label{eq:Int2}	 & \geq  r^*(\eta)^2 -\underbrace{\sup_{\substack{ \| \beta  - \beta^* \|_2 \leq 1 \\ \| \Sigma^{1/2}(\beta - \beta^*) \|_2  = r^*(\eta)  }}\bigg|  \frac{1}{n}   \sum_{i=1}^n \langle X_i,  \beta - \beta^*\rangle^2  - \bE \langle X_i,  \beta - \beta^*\rangle^2 \bigg|}_{\star} \enspace.
	\end{align}
	Finally, from Lemmas~\ref{generic_deviation} and~\ref{gaussian_mean_width} and the definition of $r^*(\eta)$, we have
	\begin{align*}
	\star & \leq c \left(   \frac{ \sum_{i=1}^p  \lambda_i(\Sigma)  \wedge  r^*(\eta)^2 }{n} + r^*(\eta) \sqrt \frac{ \sum_{i=1}^p  \lambda_i(\Sigma)  \wedge  r^*(\eta)^2 }{n} + r^*(\eta)^2 \sqrt{\frac{\log(1/\delta)}{n}}  \right) \\
	& \leq c \eta r^*(\eta)^2  + c \sqrt \eta r^*(\eta)^2 + c r^*(\eta)^2 \sqrt{\frac{\log(1/\delta)}{n}}  \leq \frac{r^*(\eta)^2}{2} \enspace.
	\end{align*}
	The last inequality holds if $\eta$ is small enough and $\delta = \exp(-c_1n)$ with $c_1>0$ small enough. Putting this bound into \eqref{eq:Int2} yields 
	\[
	\inf_{\substack{ \| \beta  - \beta^* \|_2 \leq 1 \\ \| \Sigma^{1/2}(\beta - \beta^*) \|_2  = r^*(\eta)  }}  \frac{1}{n}   \sum_{i=1}^n \langle X_i,  \beta - \beta^* \rangle^2\geqslant \frac{r^*(\eta)^2}{2}\enspace.
	\]
	Together with \eqref{eq:Int1}, this finally leads to
	$$
	 \| \Sigma^{1/2}(\hat \beta - \beta^* ) \|_2^2\leq 2\frac{\| \xi \|_2^2}{n}\enspace.
	$$
\end{proof}

\subsection{Proof of Corollaries~\ref{cor:CompBartlett} and~\ref{cor:lower_bound}}\label{sec:ProofCorollaries}
\paragraph{Proof of Corollary~\ref{cor:CompBartlett}}
	For any $r>0$ and $k = \lfloor \eta n/2 \rfloor  := cn$ ,
	$$
	\sum_{i=1}^p \lambda_i(\Sigma) \wedge r^2  \leq r_{k}(\Sigma)  + \big( \lfloor \eta n/2 \rfloor  \big) r^2 \leq   r_{k}(\Sigma)  + \big(\eta n/2 \big) r^2 \enspace,
	$$
	and it follows that
	\[
	\sum_{i=1}^p \lambda_i(\Sigma) \wedge r^2\leq \eta n r^2,\qquad \text{if}\qquad r^2\geq \frac2{\eta}\frac{r_{k}(\Sigma)}{ n}\enspace.
	\]
	Hence,
	$$
	r^*(\eta)^2 \lesssim \frac{r_{cn}(\Sigma)}{n} \enspace.
	$$
	Thus,
	
\[
\rho^2 r^*(\eta)^2 \leq \bigg(\frac{\| \beta^* \|_2^2 r_{cn}(\Sigma)}{n} \bigg) \vee \bigg( \frac{\| \xi \|_2^2 r_{cn}(\Sigma)}{n r_{k*}(\Sigma)} \bigg) \leq \bigg(\frac{\| \beta^* \|_2^2 r_{cn}(\Sigma)}n \bigg) \vee \bigg( \frac{\| \xi \|_2^2}n \bigg)\enspace. 
\]
Therefore the result follows from Theorem~\ref{theorem_loss_quadra}
\paragraph{Prood Corollary~\ref{cor:lower_bound}}
Assume that $\| \beta^* \|_2^2/ \| \xi \|_2^2 \leq 1/r_{k^*}(\Sigma) $. Thus, we have $\rho \lesssim \| \xi \|_2 / \sqrt{r_{k^*}(\Sigma) }$. For any $r>0$, from the definition of $\bar k$ given in~\eqref{eq_kbar} 
$$
\sum_{i=1}^p \lambda_i(\Sigma) \rho^2  \wedge r^2   \lesssim \| \xi \|_2^2 \frac{r_{\bar k}(\Sigma) }{r_{k^*}(\Sigma) }  + \bar k r^2 \leq \frac{\gamma}{2}  \| \xi \|_2^2 + \bar k r^2 \enspace,
$$
and it follows that $ \bar r^2(\gamma) \gtrsim \| \xi \|_2^2/  (\gamma \bar k) $ and therefore the result follows from Theorem~\ref{theorem_loss_quadra}.

\appendix
\section{Supplementary material}

\subsection{Sub-exponential random variables: definitions and properties}
The following definition and propositions can be found in~\cite{wainwright2019high}. 
\begin{Definition} \label{def_subexp} 
	A random variable $X$ with mean $\bE [X] = \mu$ is called sub-exponential with non-negative parameters $(\nu,b)$ if
	\begin{equation} 
	\bE \big [  e^{\lambda(X-\mu)}   \big] \leq e^{\nu^2 \lambda^2/2} \quad \textnormal{for all} \quad |\lambda| \leq 1/b\enspace.
	\end{equation}
\end{Definition}

\begin{Proposition} \label{proposition_sum_subexp}
	Let $X_1, \cdots, X_n$ be independent random variables such that $X_i$ is sub-exponential with parameters $(\nu_i,b_i)$. Then $Y = \sum_{i=1}^n X_i$ is sub-exponential with parameters $\big( (\sum_{i=1}^n \nu_i^2)^{1/2}, \max_{i =1,\cdots,n} b_i \big)$.
\end{Proposition}

\begin{Proposition} [Sub-exponential tail bound] \label{proposition_tail_subexp}
	Suppose that $X$ is sub-exponential with parameters $(\nu,b)$. Then
	\begin{equation}
	\bP \big(  |X-\mu| \geq t  \big) \leq \left\{
	\begin{array}{ll}
	2 e^{-t^2/(2\nu^2)} & \mbox{if } 0 < t \leq \nu^2/b \enspace,\\
	2 e^{-t/(2b)} & \mbox{if } t \geq \nu^2/b\enspace.
	\end{array}
	\right.
	\end{equation}
\end{Proposition}

\begin{footnotesize}
	\bibliographystyle{amsplain}
	\bibliography{biblio}
\end{footnotesize} 

\end{document}